\documentclass[11pt, reqno]{amsart}

\usepackage{amsthm}
\usepackage{amsfonts}
\usepackage{amssymb}
\usepackage{mathrsfs}
\usepackage[margin=1in]{geometry}
\usepackage{enumitem}
\usepackage[colorlinks]{hyperref}
\hypersetup{linkcolor=black,citecolor=blue,filecolor=black,urlcolor=black} 
\usepackage{bbm}
\usepackage[mathlines]{lineno} 
\usepackage{amsmath}           

\newtheorem{thm}{Theorem}[section]

\newtheorem{prop}[thm]{Proposition}
\newtheorem{cor}[thm]{Corollary}

\newtheorem*{main-thm}{Main Theorem}

\theoremstyle{definition}

\numberwithin{equation}{section}
\numberwithin{table}{section}

\DeclareMathOperator{\depth}{depth}         

\renewcommand{\epsilon}{\varepsilon}

\newcommand{\kk}{\mathbbm{k}}	

\DeclareMathOperator{\pd}{pd}  			
\renewcommand{\phi}{\varphi}

\DeclareMathOperator{\reg}{reg}			

\DeclareMathOperator{\Tor}{Tor}

\title[Regularity of Primes]{Prime Ideals and Three-generated Ideals with Large Regularity}
\date{\today}

\author[J. McCullough]{Jason McCullough}
\address{Iowa State University, Department of Mathematics, Ames, IA, USA}
\email{jmccullo@iastate.edu}

\begin{document}



\maketitle

\begin{abstract} Ananyan and Hochster proved the existence of a function $\Phi(m,d)$ such that any graded ideal $I$ generated by $m$ forms of degree at most $d$ in a standard graded polynomial ring satisfies $\reg(I) \le \Phi(m,d)$.  Relatedly, Caviglia et. al. proved the existence of a function $\Psi(e)$ such that any nondegenerate prime ideal $P$ of degree $e$ in a standard graded polynomial ring over an algebraically closed field satisfies $\reg(P) \le \Psi(\deg(P))$.  We provide a construction showing that both $\Phi(3,d)$ and $\Psi(e)$ must be at least doubly exponential in $d$ and $e$, respectively.  Previously known lower bounds were merely super-polynomial in both cases.
\end{abstract}

\section{Introduction}

Let $\kk$ be a field and let $S = \kk[x_1,\ldots,x_n]$ be a standard graded polynomial ring.  Let $I = (f_1,\ldots,f_m)$ be a graded ideal of $S$.  The (Castelnuovo-Mumford) regularity of $I$ is

$$\reg_S(I) = \min\{j-i\mid \Tor_i^S(I,\kk)_j \neq 0\}.  $$

Originally defined for ideal sheaves via cohomology vanishing by Mumford, this version is due to Eisenbud and Goto \cite{EG84}.  Since the maximal degree of an element in a minimal generating set of $I$ is at most $\reg(I)$, finding upper bounds on the regularity of graded ideals is a very active area of research.  Moreover, Bayer and Stillman showed that in the reverse lexicographic order and in generic coordinates, the regularity of an ideal matches that of its generic initial ideal; thus, regularity is a proxy for computation complexity.

For arbitrary ideals, regularity can be doubly exponential in terms of the degrees of the generators and number of variables.    For defining ideals of smooth projective varieties, regularity is very well-behaved \cite[Theorem 3.12]{BM93}, \cite[Corollary 4]{BEL91}.  For reduced and irreducible but possibly singular varieties, defining ideals were thought to also have well-behaved regularity, as predicted by the Eisenbud-Goto Conjecture, but counterexamples were constructed by Peeva and the author in \cite[Theorem 1.9]{MP18}.  More precisely, they showed that the regularity of a nondegenerate prime ideal over an algebraically closed field was not bounded by any polynomial function of its degree (i.e. multiplicity).  Further counterexamples were constructed in \cite[Theorem 3.3]{CCMPV19}, \cite[Theorem 1.2]{Choe22}, and \cite[3.3, 3.6, 3.8]{HK22}.  Yet, Caviglia et. al. \cite[Corollary 5.3]{CCMPV19} proved the existence of a function $\Psi(e)$ bounding the regularity of all nondegenerate primes of degree $e$.  It is easy to see that such a function does not exist even for primary ideals, thus determining the asymptotic growth rate of the optimal function $\Psi(e)$ is an interesting open problem.  In this short note, we show how to sharpen the best known lower bound for $\Psi(e)$, proving that it grows at least doubly exponentially in $e$.

Relatedly, Stillman \cite[Problem 3.15]{PS09} conjectured a bound $\Phi(m,d)$ on the regularity of a graded ideal generated by $m$ forms of degree at most $d$ without fixing the number of variables.  The existence of such a bound was first proved by Ananyan and Hochster \cite[Theorem D]{AH20} and later by Erman, Sam, and Snowden \cite[Theorem 4.12]{ESS19}.  The previous best lower bounds in the three-generated case for $\Phi(3,d)$ were quadratic \cite{Caviglia04} or super-polynomial in $d$ \cite{Choe22}.  We show that $\Phi(3,d)$ must also be at least doubly exponential in $d$.

\section{Main Results}

Three-generated ideals capture almost all of the pathologies seen in arbitrary resolutions, as is made precise by Bruns's Theorem \cite[Korollar 1]{Bruns76}.  It is known that three-generated ideals can have exponentially large projective dimension in terms of the degrees of the generators \cite[Corollary 3.6]{BMNSSS11}.  For many years, a construction due to Caviglia \cite[Example 4.2.1]{Caviglia04} of a three-generated ideal with quadratic growth of regularity was the best known lower bound for $\Phi(3,d)$.  As Eisenbud writes in \cite[p. 62]{Eisenbud05}, ``It would be interesting to have more and stronger examples with high regularity.'' The author conjectured that a certain three-generated ideal family exhibited super-polynomial growth of regularity relative to the degrees of the three generators \cite[Conjecture 13.9]{MP20}.  Recently, Choe \cite[Theorem 1.2]{Choe22} constructed a different such example.  Below we construct a three-generated ideal with doubly exponential growth of its regularity in terms of the common degree of its minimal generators.

The following is a slight generalization of \cite[Proposition 5.1]{Choe22}, which in turn generalized \cite[Proposition 7.1]{CCMPV19}.  It shows how to compute invariants of three-generated analogues of arbitrary graded ideals.  We include a proof for completeness. 

\begin{prop}\label{3gen} Let $I = (g_0,\ldots,g_m)$ be a graded ideal of a standard graded polynomial ring $S = \kk[x_1,\ldots,x_n]$ with $d_i = \deg(g_i)$ nondecreasing and $d = d_m$.  Let $R = S[y,z]$ and pick integers $s,t \ge m + d - d_0 + 1$.  Set 

$$J = \left(y^s, z^t, \sum_{i = 0}^m g_i y^{d-d_i+i} z^{m-i}\right) \subseteq R.$$  Then $J$ is a graded ideal with generators of degree $s,t,m+d$.  Moreover, 
\begin{enumerate}
    \item $\reg_R(J) \ge \reg_S(I) + s + t -2$,
    \item $\pd_R(J) \ge \pd_S(I) + 2$,
    \item $\deg(J) = tm.$    
\end{enumerate}
\end{prop}

\begin{proof} Note that $R/J$ is a finitely generated $S$-module and that $y^{s-1}z^{t-1}(R/J) \cong (S/I)(-s - t+2)$ as an $S$-direct summand of $R/J$.  Thus

$$\reg_R(R/J) \ge \reg_S((S/I)(-s-t+2)) = \reg_S(S/I) +s+t - 2,   $$

or equivalently, 

$$\reg_R(J) \ge \reg_S(I) + s + t - 2.$$

Similarly, suppose $\pd_S(S/I) = p$.  Modding out by a regular sequence of linear forms (extending the field if necessary), we may assume that $\depth(S/I) = 0$ so that $S/I$ has a nonzero socle element $w$.  Clearly $y^{s-1}z^{t-1}w$ is a nonzero socle element in $R/J$.  The result follows by the Auslander-Buchsbaum Theorem.

For the degree computation, simply note that since $\sqrt{J} = (y,z) =\mathfrak{p}$, we have $\deg(J) = \lambda(S_{\mathfrak{p}}/J_{\mathfrak{p}})$ and $S_\mathfrak{p}/J_\mathfrak{p}$ has $S_\mathfrak{p}/\mathfrak{p}S_\mathfrak{p}$-basis $y^iz^j$ for $0 \le i \le m-1$ and $0 \le j \le t-1$.
\end{proof}

As an immediate corollary, we see that finding an explicit answer to Stillman's Conjecture reduces to the three-generated case.

\begin{cor} To find an explicit function $\Phi(m,d)$ bounding the regularity of ideals generated by $m$ forms of degree at most $d$, it suffices to find an explicit bound $\Phi(3,d)$ in the three-generated ideal case.
\end{cor}

We also get a larger lower bound on the possible value of $\Phi(3,d)$.  Recall that a function $f:\mathbb{N} \to \mathbb{N}$ is $\Omega(g(n))$ if there is a constant $c$ such that $f(n) \ge c \cdot g(n)$ for $n \gg 0$.

\begin{thm}\label{3reg}
    Let $\kk$ be a field and fix an integer $r \ge 2$.  Then there is a standard graded polynomial ring $R$ over $\kk$ and a graded ideal $J$ of $R$ generated by $3$ forms of degree $22r-2$ with $\reg(J) \ge 2^{2^{r-1}} + 44r - 6$.  In particular, $\Phi(3,d)$ is $2^{2^{\Omega(d)}}$.
\end{thm}

\begin{proof}
    By \cite{Koh98}, there is a polynomial ring $S$ and graded ideal $I$ generated by $22r-3$ quadrics with $\reg(I) \ge 2^{2^{r-1}}$.  Applying Proposition~\ref{3gen} with $m = 22r-4$ and $s = t = 22r-2$, we get an ideal $J$ generated by three forms of degree $22r-2$ and $\reg(J) \ge 2^{2^{r-1}} + 44r - 6$.
\end{proof}

We turn now to constructing prime ideals with large regularity.  In \cite{MP18}, Peeva and the author employed Rees-like algebras to create counterexamples of the Eisenbud-Goto Conjecture.  After homogenizing the defining ideals of such algebras, they computed the invariants of the new standard graded prime ideals in terms of the arbitrary starting ideal.  We record the essential part of their theorem here.

\begin{thm}[{McCullough-Peeva \cite{MP18}}]\label{reeslike}
    Let $\kk$ be any field. Let $I = (f_1,\ldots, f_m)$ be a graded ideal in the standard graded polynomial ring $S = \kk[x_1,\ldots,x_n]$.
There is a nondegenerate graded prime ideal $P$ in the standard graded polynomial ring
$R = S[y_1,\ldots,y_m,u_1,\ldots,u_m,z,v]$
satisfying
\begin{enumerate}
    \item $\reg_R(P) = \reg_S(I) + 2 + \sum_{i = 1}^m \deg(f_i)$,
    \item $\deg_R(R/P) = 2 \prod_{i = 1}^m (\deg(f_i) + 1)$.
\end{enumerate}
\end{thm}

The trick to achieving more extreme counterexamples to the Eisenbud-Goto Conjecture is to feed the Mayr-Meyer ideals, such as Koh's version \cite{Koh98}, into Proposition~\ref{3gen}, and moreover to do it twice, before constructing a Rees-like algebra.

\begin{thm} Over any field $\kk$ and for any integer $r$, there exists a standard graded polynomial ring $S$ and nondegenerate prime ideal $P \subseteq S$ with

$$\deg(P) = 396r \qquad \text{ and } \qquad \reg(P) \ge 2^{2^{r-1}} + 66r + 6.$$  In particular, $\Psi(e)$ is $2^{2^{\Omega(e)}}$.
\end{thm}

\begin{proof}
    Begin with the ideal $J$ from Theorem~\ref{3reg} which has three generators of degree $22r-2$ and $\reg(J) \ge 2^{2^{r-1}} + 44r - 6$.  Applying Proposition~\ref{3gen} again with $m = 2$ and $s = t = 3$, we get an ideal $J'$ generated by forms of degrees $3,3,22r$ with $\reg(J') \ge 2^{2^{r-1}} + 44r-2$.  Applying Thoerem~\ref{reeslike} to $J'$ yields a standard graded prime ideal $P$ with $\deg(P) = 396r$ and $\reg(P) \ge 2^{2^{r-1}} + 66r +6$.
\end{proof}

In particular, we get counterexamples to the Eisenbud-Goto Conjecture for $r \ge 5$.  

We recall that the existence of a function $\Phi(m,d)$ bounding the regularity of an ideal generated by $m$ forms of degree at most $d$ is equivalent to a function $\Psi(e)$ bounding the regularity of nondegenerate prime ideals of degree $e$.  Given $\Phi(m,d)$, it follows from \cite{CCMPV19} that 

$$\Psi(e) \le \max\{e^2,\Phi(e,e^2)+1\}.$$

Similarly, given an explicit function $\Psi(e)$, it follows from \cite{MP18} that

$$\Phi(m,d) \le \Psi(2(d+1)^m).$$

There is a parallel version of Stillman's Conjecture \cite{PS09} which posits a bound $\Phi_{\mathrm{pd}}(m,d)$ on the projective dimension of graded ideals generated by $m$ forms of degree at most $d$.  Caviglia showed that the existence of $\Phi(m,d)$ is equivalent to that of $\Phi_{\mathrm{pd}}(m,d)$.  In fact, he showed that

$$\Phi(m,d) \le (2d)^{2^{\Phi_{\mathrm{pd}}(m,d)-2}}$$

and

$$\Phi_{\mathrm{pd}}(m,d) \le \Phi(m,d) \cdot \sum_{i = 0}^{\Phi(m,d)} m^{2i}.$$

See \cite[Theorem 4]{MS13} for details.  
There is yet a fourth function $\Psi_{\mathrm{pd}}(e)$ bounding the projective dimension of nondegenerate prime ideals of degree $e$ over an algebraically closed field.  Given $\Phi_{\mathrm{pd}}(m,d)$, it follows from \cite[Theorem 5.2]{CCMPV19} that

$$\Psi_{\mathrm{pd}}(e) \le \max\{e, \Phi_{\mathrm{pd}}(e,e^2)-1\}.$$

Similarly, given $\Psi_{\mathrm{pd}}(e)$, it follows from \cite[Theorem 1.6]{MP18} that

$$\Phi_{\mathrm{pd}}(m,d) \le \Psi_{\mathrm{pd}}(2(d+1)^m).$$

Thus finding explicit formulas for any of $\Psi(e), \Psi_{\mathrm{pd}}(e), \Phi(m,d), \Phi_{\mathrm{pd}}(m,d)$ will yield explicit formulas for all of the others.  Currently the best lower bound estimate for $\Phi_{\mathrm{pd}}(3,d)$ is $\sqrt{d}^{\sqrt{d}-1}$ \cite[Corollary 3.6]{BMNSSS11}.  Using Theorem~\ref{3gen}, we can increase this lower bound as well.

\begin{thm}
    For any integer $r \ge 1$, there is a standard graded ideal $I$ with three generators in degree $4r+1$ with $\mathrm{pd}(I) \ge r^{2r}$.  In particular, $\Phi_{\mathrm{pd}}(3,d)$ is $(d/5)^{\Omega(d)}$.
\end{thm}

\begin{proof}
    By \cite[Corollary 3.7]{BMNSSS11}, there is an ideal $I$ with $2r+1$ generators of degree $2r+1$ and $\pd_S(S/I) \ge r^{2r}$.  Applying Proposition~\ref{3gen} with $m = 2r$ and $s = t = 4r+1$ yields an ideal $J$ generated by three forms of degree $4r+1$ and $\pd_R(J) \ge r^{2r} + 1$.
\end{proof}

\section*{Acknowlegements} 

This work was completed while preparing for the Workshop on Regularity and Syzygies hosted at the University of Illinois Chicago and organized by Wenliang Zhang.  The author thanks Paolo Mantero, Matt Mastroeni, and Lance Edward Miller for feedback on an earlier draft of this paper.  The author was support by National Science Foundation grant DMS--1900792.

\bibliographystyle{plainurl}
\bibliography{reg}

\end{document}